\documentclass{amsart}
\usepackage{amssymb, amsmath, amsthm, enumerate, tikz, float}

\def\nor{\trianglelefteq}
\def\ord#1{\vert #1 \vert}
\def\ZZ{\mathbb{Z}}

\def\Z#1{\textrm{Z}(#1)}

\newtheorem*{thm}{Theorem}
\newtheorem{prop}{Proposition}[section]
\newtheorem{cor}[prop]{Corollary}
\newtheorem{lem}[prop]{Lemma}
\theoremstyle{definition}
\newtheorem*{cnstr}{Construction}

\def\CD#1{\mathcal{C}\mathcal{D} (#1)}

\begin{document}

\author{Ben Brewster}
\address{Department of Mathematical Sciences, Binghamton University, Binghamton, New York, United States} 
\email{ben@math.binghamton.edu}

\author{Peter Hauck}
\address{Fachbereich Informatik, Eberhard-Karls-Universit\"at T\"ubingen, T\"ubingen, Germany}
\email{hauck@informatik.uni-tuebingen.de}

\author{Elizabeth Wilcox}
\address{Mathematics Department, Colgate University, Hamilton, New York 13346}
\email{ewilcox@colgate.edu}

\title{Groups whose Chermak-Delgado Lattice is a Chain}

\begin{abstract} For a finite group $G$ with subgroup $H$, the {\it Chermak-Delgado measure of $H$ in $G$} 
refers to $\ord H \ord {C_G(H)}$.  The set of all subgroups with maximal Chermak-Delgado measure form a sublattice, 
$\CD G$, within the subgroup lattice of $G$.  This paper examines conditions under which the Chermak-Delgado 
lattice is a chain of subgroups $H_0 < H_1 \cdots < H_n$.  
On the basis of a general result how to extend certain Chermak-Delgado lattices, we construct for any prime $p$ and 
any non-negative integer $n$ a $p$-group whose Chermak-Delgado lattice is a chain of length $n$.  \end{abstract}

\maketitle

I. Martin Isaacs \cite{Isaacs} examined a function from the set of subgroups of a finite group into the set of 
positive integers, one of a family of such functions originally defined by A. Chermak and A. Delgado \cite{cd} as a means 
to obtain results on the existence of large normal abelian subgroups.  
Isaacs referred to this function as the {\it Chermak-Delgado measure} and revisited the proof that, for a finite group $G$, 
the subgroups with maximal measure form a (self-dual) sublattice within the subgroup lattice of $G$, which Isaacs 
called the {\it Chermak-Delgado lattice} of $G$.  

It was shown in \cite{BW2012} that the Chermak-Delgado lattice of a direct product of finite groups decomposes 
as the direct product of the Chermak-Delgado lattices of the factors.  Thus, at least for finite nilpotent groups, 
the study of the Chermak-Delgado lattice rests on understanding this lattice in finite $p$-groups.  
Already in small non-abelian $p$-groups, the Chermak-Delgado lattice usually has a rather complicated structure. 
We therefore ask whether more elementary types of  self-dual lattices, like chains, can be realized as Chermak-Delgado lattices 
of $p$-groups. Clearly, the Chermak-Delgado 
lattice of any finite abelian group is a chain of length 0. We prove that for any prime $p$ and any positive 
integer $n$ there exists a $p$-group of class 2 whose Chermak-Delgado lattice is a chain of length $n$.

We obtain this result by first constructing $p$-groups (of smallest possible order) with Chermak-Delgado 
lattice a chain of length 1 or 2. The general statement then follows from a construction which embeds 
certain $p$-groups $H$ of class 2 in $p$-groups $G$ of class 2 such that the Chermak-Delgado lattice 
of $G$ extends the Chermak-Delgado lattice of $H$ by a new minimal and a new maximal element.

\section{Preliminaries}\label{prelim}

Let $G$ be a finite group.  If $H \leq G$ then the {\it Chermak-Delgado measure of $H$ (in $G$)} 
is denoted by $m_G(H)$, and defined as 
\begin{equation*}
m_G(H) = \ord H \ord {C_G(H)}.
\end{equation*}
Define
\begin{equation*}
m(G) = \textrm{max}\{m_G(H) \mid H \leq G \} \qquad \textrm{and} \qquad \CD G = \{ H \mid m_G(H) = m(G)\}.
\end{equation*}
The set $\CD G$ is the Chermak-Delgado lattice of $G$.  The fact that $\CD G$ is a sublattice of the subgroup 
lattice of $G$ was determined in \cite{cd}.  We refer the reader to \cite[Section 1G]{Isaacs} for an introduction to the 
topic as well as the proofs of the following facts:	

\begin{prop}\label{cdomnibus} Let $G$ be a finite group.  Then $\CD G$ is a sublattice in the subgroup lattice 
of $G$ with the following properties:
	\begin{enumerate}
	\item $\langle H, K \rangle = HK$ for $H, K \in \CD G$.
	\item If $H \in \CD G$, then $C_G(H) \in \CD G$ and $C_G(C_G(H)) = H$.
	\item The maximum subgroup in $\CD G$ is characteristic in $G$.
	\item The  minimum subgroup of $\CD G$ is characteristic, abelian, and contains $\Z G$.
	\end{enumerate}\end{prop}

\begin{prop}\label{maxmember} Let $G$ be a finite group.  If $M$ is the maximal member in $\CD G$ then $\CD G = \CD M$.\end{prop}
  
\noindent
Proposition~\ref{maxmember} is straightforward to prove, relying on the fact that if $H \in \CD G$ then $C_G(H) = C_M(H)$.  We shall use these statements frequently throughout the paper without any further reference.\\

For a non-negative integer $n$, we call a series of subgroups $G_0 < \ldots < G_n$ of a group $G$ a {\it chain of length $n$}.  Examples of groups $G$ where $\CD G$ is a chain of length 0 or 1 abound -- $G$ abelian or $G = S_n$ for $n \geq 4$, respectively -- but examples with longer chains take some work to produce.  Moreover, the readily-available examples $G$ where $\CD G$ 
is a chain of length 1 have composite order.  We therefore restrict ourselves in this paper to $p$-groups whose 
Chermak-Delgado lattice is a chain.\\

The following simple lemma is needed frequently in Section~\ref{short} when we determine the structure of $p$-groups of small order whose Chermak-Delgado lattice is a chain of length 1 or 2.

\begin{lem}\label{centr} Let $p$ be a prime. Suppose $P$ is a $p$-group with normal subgroups $R$, $Q$ such that $\Z P < R \leq Q$ and $\ord{R/\Z P} = p$. If $\ord P > \ord Q\ord{[R,P]}$ then there exists $x \in P \setminus Q$ such that $\ord{C_P(x)} \ge p^2\ord{\Z P}$.
\end{lem}

\begin{proof} Note that $[R,P] \le \Z P$. Let $x_1,\ldots,x_n$ be coset representatives of $Q$ in $P$. 
Let $y \in R \setminus \Z P$. By hypothesis, 
$n > \ord{[R,P]}$. Hence there exist $i$, $j$ with $i \ne j$ such that $[y,x_i] = [y,x_j]$, whence $y \in C_P(x_ix_j^{-1})$. Since 
$x:= x_ix_j^{-1} \not\in Q$, it follows that $\ord{C_P(x)} \ge \ord{\langle x\rangle R} \ge p^2\ord{\Z P}$.
\end{proof}

\section{Chains of Length 1 or 2}\label{short}

Let $p$ be a prime. The aim of this section is to show that there are non-abelian groups $P$ of order $p^6$ with 
$\CD P = \{\Z P, P\}$; in fact, for odd $p$ there are even groups of order $p^5$ with this property, but not for 
$p = 2$. No $p$-groups of smaller order have a chain of length 1 as Chermak-Delgado lattice. Similarly, 
chains of length 2 occur as Chermak-Delgado lattices among groups of order $p^7$, and for odd $p$ even 
for certain groups of order $p^6$ but not for $p=2$ or $p$-groups of smaller order.

\begin{prop}\label{l1} Let $p$ be a prime and $P$ a $p$-group of order at most $p^5$. Then $\CD P$ is a chain of 
length 1 if and only if  $p \ne 2$ and $\ord P = p^5$, $\ord{\Z P} = p^2$, $P/\Z P$ is extraspecial of exponent $p$ 
and $[Z_2(P),P] = \Z P$.
\end{prop} 

\begin{proof} Assume that $\CD P$ is a chain of length 1. Clearly $P$ is non-abelian and we may assume that $\CD P = \{\Z P, P\}$. We show first for all $x \in P / \Z P$ that 
\begin{equation}
C_P(x) = \langle x \rangle \Z P  \quad \textrm{and} \quad \ord{C_P(x)} = p \cdot \ord{\Z P}\tag{*}.
\end{equation} 
Notice $P$ does not contain an abelian subgroup $D$ of order $p^2 \cdot \ord{\Z P}$; otherwise 
$m_P(D) \geq p^4 \cdot \ord{\Z P}^2 \geq p^5 \cdot \ord{\Z P} = m(P)$, contradicting $\CD P = \{\Z P, P\}$. 
In particular, $\ord{\langle x \rangle \Z P} = p \cdot \ord{\Z P}$ for $x \in P \setminus \Z P$.
If $C_P(x) > \langle x \rangle \Z P$ then $C_P(x)$ contains an abelian subgroup $D$ 
of order $p^2 \cdot \ord{\Z P}$ since $\langle x \rangle \Z P \le \Z{C_P(x)}$. This contradiction implies (*).

It is clear that $\ord P > p^3$ since in an extraspecial group of order $p^3$ the $p+1$ maximal abelian subgroups are 
in $\CD P$. Moreover, application of Lemma~\ref{centr} with an arbitrary normal subgroup $Q = R \le Z_2(P)$ shows that the assumption $\ord{\Z P} = p$ leads to a contradiction against (*). Hence $\ord{\Z P} > p$.

If $P$ has order $p^4$ then $\ord{\Z P} = p^2$; therefore $P$ has an abelian subgroup $A$ of order $p^3$ whence 
$m_P(A) \ge p^6 = m(P)$. It follows that $\ord P = p^5$. 
If $\ord{\Z P} = p^3$ then $m(P) = p^8 = m_P(A)$ for any subgroup $A$ of order $p^4$ containing $\Z P$, contradicting the 
assumption that $\CD P = \{\Z P, P\}$. Hence $\ord{\Z P} = p^2$. 

By (*), $P/ \Z P$ has exponent $p$. If $P/\Z P$ is elementary abelian,  let $Z$ be a subgroup of order $p$ in $\Z P$. By \cite[III, Satz 13.7]{Huppert}, $\Z{P/Z} > \Z P/Z$. Hence 
there exists $x \in P \setminus \Z P$ such that $[x,P] = Z$. Then, contrary to (*), $\ord{C_P(x)} = p^4$ since $g \mapsto [x,g]$ is a homomorphism from $P$ onto $Z$ with $C_P(x)$ as its kernel.

Thus $P/\Z P$ is not elementary abelian. Therefore, $p \ne 2$ and $P/\Z P$ is extraspecial of exponent $p$.  
If $[Z_2(P),P]$ has order $p$, then Lemma~\ref{centr} 
(with $Q = R = Z_2(P)$) contradicts (*). Thus $[Z_2(P),P] =\Z P$.\\

Assume now that $p \ne 2$, $\ord P = p^5$, $\ord{\Z P} = p^2$, $P/\Z P$ is extraspecial of exponent $p$ 
and $[Z_2(P),P] = \Z P$. Let $\langle t\rangle \Z P = Z_2(P)$ and choose $x_1,x_2 \in P \setminus Z_2(P)$ such that 
$[x_1,x_2]\Z P = t\Z P$. Note that $[x_i,t]^p = [x_i^p,t] = 1$ for $i = 1,2$. Therefore 
$[Z_2(P),P] = \Z P$ implies that $\Z P = \langle[x_1,t]\rangle \times \langle [x_2,t]\rangle$, hence $\Z P$ is elementary abelian of order $p^2$. 
From this it follows immediately that $C_P(x) = \langle x\rangle \Z P$ for all $x \in P \setminus \Z P$. Consequently, 
every subgroup of order $p^3$ containing $\Z P$ is self-centralizing and the centralizer of every subgroup of order 
$p^4$ is $\Z P$. It follows that for each such subgroup the Chermak-Delgado measure is $p^6$ whereas 
$m_P(\Z P) = m_P(P) = p^7$. Hence $\CD P = \{\Z P, P\}$ as claimed.
\end{proof}

\begin{cor}\label{l1odd} Let $p$ be an odd prime. Then there exists a non-abelian $p$-group $P$ of order $p^5$ with 
$\CD P = \{\Z P, P\}$.
\end{cor}

\begin{proof} Let $P$ be generated by $x_1,x_2,t,z_1,z_2$ according to the following defining relations:
\begin{equation*}
x_1^p = x_2^p = t^p = z_1^p = z_2^p = 1,\end{equation*}
\begin{equation*}
[x_1,x_2] = t, [t,x_1] = z_1, [t,x_2] = z_2,
\end{equation*}
\begin{center} all other commutators between the generators equal 1.\end{center}
Clearly $t, z_1, z_2$ generate an elementary abelian group of order $p^3$ and $x_1$ induces an automorphism 
of order $p$ on this group. We claim that the relations for $x_2$ define an automorphism of order $p$ on 
$\langle x_1,t,z_1,z_2\rangle$. The relation $x_1^p = 1$ is preserved under the action of $x_2$ since 
$(x_1t)^p = x_1^pt^pz_1^{\binom{p}{2}} = 1$ as $p$ is odd. All other relations in $\langle x_1,t,z_1,z_2\rangle$ 
are trivially preserved. That $x_2$ has actually order $p$ follows essentially from the fact that 
$x_1^{x_2^p} = x_1t^pz_2^{\binom{p}{2}} = x_1$, again because $p$ is odd.

Hence $\ord P = p^5$, $\Z P = \langle z_1,z_2\rangle = \langle [t,P]\rangle$ has order $p^2$ and $P/\Z P$ 
is extraspecial of exponent $p$. The assertion follows now from Proposition~\ref{l1}. 
\end{proof}

We show now that there is a group of order $2^6$ whose Chermak-Delgado lattice is a chain of length 1. The 
corresponding construction is possible for any $p$ (and we present it this way) and leads to groups $P$ 
with $P/\Z P$ elementary abelian (in contrast to the $p$-groups of order $p^5$ for odd $p$ with chains of 
length 1 as Chermak-Delgado lattices; cf. Proposition~\ref{l1}). This property is in fact needed later 
when we apply the extension theorem (Theorem~\ref{extthm}) to obtain $p$-groups with Chermak-Delgado 
lattices being chains of arbitrary length.

\begin{prop}\label{l1n} Let $p$ be a prime. There exists a group $P$ of order $p^6$ such that $\Z P$ and 
$P/\Z P$ are elementary abelian of order $p^3$ and $\CD P = \{\Z P, P\}$.
\end{prop}

\begin{proof} Let $P$ be generated by $x_1,x_2,x_3,z_{1,2},z_{1,3},z_{2,3}$ subject to the following defining relations:
\begin{equation*}
x_1^p = x_2^p = x_3^p = z_{1,2}^p = z_{1,3}^p = z_{2,3}^p = 1,
\end{equation*}
\begin{equation*}
[x_1,x_2] = z_{1,2}, [x_1,x_3] = z_{1,3}, [x_2,x_3] = z_{2,3},
\end{equation*}
\begin{center} all other commutators between the generators equal 1.\end{center}
It is clear that $\ord P = p^6$, $\Z P = \langle z_{1,2}, z_{1,3}, z_{2,3}\rangle$, $m_P(P) = m_P(\Z P) = p^9$. 
It is also straightforward to verify that $C_P(x) = \langle x\rangle \Z P$ for all $x \in P \setminus \Z P$. 
This implies that $m_P(U) = p^8$ for all $\Z P < U < P$ and the assertion follows.
\end{proof}

We now turn to $p$-groups whose Chermak-Delgado lattice is a chain of length 2.

\begin{prop}\label{l2} Let $p$ be a prime and $P$ be a $p$-group of order at most $p^6$. Then $\CD P$ is a chain of 
length 2 if and only if  $p \ne 2$ and $\ord P = p^6$, $\ord{\Z P} = p^2$, $\ord{P'} = p^3$, there exists an abelian 
normal subgroup $A$ of $P$, $\ord A = p^4$ and $[A,x] = \Z P$ for all $x \in P \setminus A$.

In this case, $\CD P = \{\Z P, A, P\}$. 
\end{prop}

\begin{proof} Suppose $\CD P$ is a chain of length 2: we may assume by Proposition~\ref{maxmember} there exists a non-central subgroup $A < P$ such that $\CD P = \{\Z P, A, P\}$. Since $C_P(A)$ and all conjugates 
of $A$ are in $\CD P$, $A$ is an abelian self-centralizing normal subgroup of $P$. In particular, 
$m(P) = \ord{\Z P}\ord P = \ord A^2$ and every abelian subgroup distinct from $A$ has order less 
than $\ord A$. 

Suppose that $\ord P \le p^5$. As in the proof of Proposition~\ref{l1} we may assume that $\ord P = p^5$ and 
$P \in \CD P$. Since $m(P)$ is a square, $\ord{\Z P}$ equals $p$ or $p^3$. In the latter case, $P/\Z P$ is elementary abelian 
of order $p^2$, whence there are $p+1$ abelian subgroups of order $p^4 = \ord A$ containing $\Z P$. This contradiction implies $\ord{\Z P} = p$. If $R$ is a normal subgroup of $P$ of order $p^2$ contained in $A \cap Z_2(P)$ then 
Lemma~\ref{centr} (with $Q = A$) yields an element $x \in P \setminus A$ whose centralizer has order at least $p^3$. Hence $P$ has an abelian subgroup containing $\langle x\rangle \Z P$ of order at least $p^3 = \ord A$, contrary to the conclusion in the first paragraph.

Therefore $\ord P = p^6$. Since $\ord{\Z P} = p^4$ is clearly 
impossible, it follows from $m(P) = \ord A^2$ that $\ord{\Z P} = p^2$ and $\ord A = p^4$. From the conclusions of the first paragraph in the proof, we deduce that $C_P(x) = \langle x\rangle\Z P$ has order $p^3$ for all 
$x \in P \setminus A$. In particular, every element in $P/\Z P \setminus A/\Z P$ has order $p$.

We show next that $[A,P] \le \Z P$. Suppose there exists $a \in A$ such that $[a,x] \not\in \Z P$ for 
all $x \in P \setminus A$. Let $x_1,\ldots,x_{p^2}$ be coset representatives of $A$ in $P$. If there 
exist $i$, $j$ with $i \ne j$ such that $[a,x_i]\Z P = [a,x_j]\Z P$, then $[a,x_ix_j^{-1}] \in \Z P$ (note that 
$[A,P] \le Z_2(P)$). Since $x_ix_j^{-1} \not\in A$, this contradicts our assumption. It follows 
that  $\{ [a,x_i] \mid i = 1,\ldots,p^2\}$ constitutes a complete set of coset representatives for $\Z P$ in 
$A$. But then $[A,P] = \bigcup_{i=1}^{p^2} [[a,x_i],P] \le \Z P$, contradicting our assumption.

Consequently for every $a \in A$ there exists $x \in P \setminus A$ such that $[a,x] \in \Z P$. Given 
any $y \in P \setminus A$ it follows from $[x,y] \in A$ and the Witt Identity that $[a,y]$ centralizes $x$. 
As we have noted before, $C_A(x) = \Z P$, whence $[a,y] \in \Z P$. Therefore, $[A,P] \le \Z P$.

Suppose there is some $x \in P \setminus A$ such that $\ord{[A,x]} = p$. Considering the homomorphism 
$a \mapsto [a,x]$ for $a \in A$ leads to $\ord{C_A(x)} = p^3$, a contradiction. Hence $[A,x] = \Z P$ for 
all $x \in P \setminus A$.

Let $x \in P \setminus A$ and $a \in A$. Since $x^p \in \Z P$ and $[a,x] \in \Z P$, it follows 
that $1 = [x^p,a] = [x,a^p]$ whence $a^p \in C_A(x) = \Z P$. Therefore $A$ is elementary abelian 
and hence $P/\Z P$ has exponent $p$.

Suppose that $P/\Z P$ is elementary abelian. If $x \in P \setminus A$ then $\langle x\rangle \Z P$ is 
normal in $P$. It follows from Lemma~\ref{centr} (with $Q = R = \langle x\rangle \Z P$) that 
$C_P(x)$ has order at least $p^4$. This contradiction implies $P/\Z P$ cannot be elementary abelian and, since it 
has exponent $p$, therefore $p$ is odd.

Choose $x,y \in P$ such that $P = A\langle x, y\rangle$. Since $P' \le A \le Z_2(P)$, it follows that 
$[\langle x \rangle, \langle y \rangle]\Z P = \langle [x,y]\rangle\Z P$. As $[A,P] = \Z P$, this implies 
that $P' = \langle [x,y]\rangle\Z P$. Hence $\ord{P'} = p^3$ as $P/ \Z P$ is non-abelian of exponent $p$. 
(We note that $\ord{P'} = p^3$ can also be inferred from \cite[Lemma 14]{PS}).  

This proves one implication of the proposition.\\

For the converse assume that $P$ is a group with the properties listed in the proposition; we show 
$\CD P = \{\Z P, A, P\}$.

Let $x \in P \setminus A$. Since $[A,x] = \Z P$, it follows that $C_A(x)$ has order $p^2$ whence $C_A(x) = \Z P$. 
Now $\ord{C_P(x)} = p^4$ would imply that $C_P(x)$ is abelian and $P = C_P(x)A$, but then $P' = [C_P(x),A] = \Z P$ 
as $[A,P] = [A,x] = \Z P$ by hypothesis. This contradicts $\ord{P'} = p^3$. Hence $C_P(x) = \langle x\rangle\Z P$ and 
$x^p \in \Z P$ for all $x \in P \setminus A$. Also, 
any $a \in A \setminus \Z P$ is not centralized by an element outside $A$ whence $C_P(a) = C_P(A) = A$. 
This implies that the centralizer of a subgroup $B$ of order $p^3$ has order $p^3$ if $B$ is not contained in $A$ 
and order $p^4$ otherwise. If $B \ne A$ has order $p^4$ then it follows immediately (by considering the cases 
$\ord{A \cap B} = p^2$ or $p^3$) that $C_P(B) = \Z P$; thus the same holds for subgroups of order $p^5$. 

It follows that $m_P(B) < p^8 = m_P(\Z P) = m_P(A) = m_P(P)$ for all subgroups $B \ge \Z P$ different from $\Z P$, $A$, or $P$. 
This proves the second direction of the proposition.
\end{proof}

We remark that some statements of Propositions~\ref{l1} and \ref{l2} could have been deduced from \cite{PS} but this would 
not have shortened the straightforward proofs given here.

\begin{cor}\label{l2odd} Let $p$ be an odd prime. Then there exists a $p$-group $P$ of order $p^6$ whose 
Chermak-Delgado lattice is a chain $\Z P < A < P$ of length 2.
\end{cor}

\begin{proof} Since the map $f: \ZZ_p \longrightarrow \ZZ_p$ defined by $f(m) = m(m+1)\bmod p$ is not 
surjective, there exists $c \in \ZZ_p$ not contained in the image of $f$. 

Let $P$ be generated by $x_1,x_2,a_1, a_2,z_1,z_2$ according to the following defining relations:
\begin{equation*}
x_1^p = x_2^p = a_1^p = a_2^p = z_1^p = z_2^p = 1,
\end{equation*}
\begin{equation*}
[x_1,x_2] = a_1, [a_1,x_1] = z_1, [a_1,x_2] = z_1z_2^c, [a_2,x_1] = z_2, [a_2,x_2] = z_1,
\end{equation*}
\begin{center} all other commutators between the generators equal 1.\end{center}
It is clear that $a_1, a_2, z_1, z_2$ generate an elementary abelian group $A$ of order $p^4$ and 
that $x_1$ induces an automorphism of order $p$ on $A$. We claim that the relations for $x_2$ define an automorphism of order $p$ on 
$\langle x_1,a_1,a_2,z_1,z_2\rangle$. The relation $x_1^p = 1$ is preserved under the action of $x_2$ since 
$(x_1a_1)^p = x_1^pa_1^pz_1^{\binom{p}{2}} = 1$ as $p$ is odd. All other relations in $\langle x_1,a_1,a_2,z_1,z_2\rangle$ 
are trivially preserved. That $x_2$ has actually order $p$ follows essentially from the fact that 
$x_1^{x_2^p} = x_1a_1^p(z_1z_2^c)^{\binom{p}{2}} = x_1$, again because $p$ is odd.

Thus $P$ has order $p^6$, $A$ is an abelian normal subgroup of order $p^4$, $P' = \langle a_1, z_1, z_2\rangle$ has 
order $p^3$ and clearly $\Z P = \langle z_1,z_2\rangle$ has order $p^2$. In light of Proposition~\ref{l2} it 
remains to be shown that $[A,x] = \Z P$ for all $x \in P \setminus A$. We may assume that $x = x_1^ix_2^j$ for $i, j$ with
$0 \le i,j \le p-1$ and $i,j$ not both equal to 0. Since 
\begin{equation*}
[a_1,x] = z_1^{i+j}z_2^{cj} \qquad \textrm{and} \qquad [a_2,x] = z_1^jz_2^i,
\end{equation*}
$[a_1,x]$ and $[a_2,x]$ do not generate $\Z P$ if and only if the determinant of the $\ZZ_p$-matrix
\begin{equation*}
\left( \begin{array}{cc}
							i + j & cj\\
							j & i\\
				\end{array}
		\right)
\end{equation*} 
equals 0, that is $i^2 + ij -cj^2 = 0$ in $\ZZ_p$. Suppose this is the case. If $j=0$ then $i=0$. Hence $j \ne 0$ whence 
$c = (ij^{-1})^2 + ij^{-1}$ in $\ZZ_p$ which is impossible by the definition of $c$. Hence $[A,x] = \Z P$ and $\CD P$ is a chain of length 2 by Proposition~\ref{l2}.
\end{proof}    

As with chains of length 1, for the construction in Section~\ref{extend} we also need $p$-groups $P$ whose 
Chermak-Delgado lattice is a chain of length 2 such that  
$P/\Z P$ is elementary abelian. We show now that for every prime $p$ there exist groups of order $p^7$ with this property. 
In particular, $2^7$ is the smallest order of a 2-group with a chain of length 2 as Chermak-Delgado lattice.

\begin{prop}\label{l2n} Let $p$ be a prime. There exists a group $P$ of order $p^7$, $\Z P$ and 
$P/\Z P$ both elementary abelian, whose 
Chermak-Delgado lattice is a chain $\Z P < A < P$ of length 2.
\end{prop}

\begin{proof} The construction is a slight variation of the one given in the preceding Proposition~\ref{l2odd}. Choose $0 < c < p$ 
as in the proof of Proposition~\ref{l2odd} and let $P$ be generated by $x_1, x_2, a_1,a_2,z,z_1,z_2$ subject to the 
following relations:
\begin{equation*}
x_1^p = x_2^p = a_1^p = a_2^p = z^p = z_1^p = z_2^p = 1,
\end{equation*}
\begin{equation*}
[x_1,x_2] = z, [a_1,x_1] = z_1, [a_1,x_2] = z_1z_2^c, [a_2,x_1] = z_2, [a_2,x_2] = z_1,
\end{equation*}
\begin{center} all other commutators between the generators equal 1.\end{center}
It is clear that $A_0 = \langle a_1, a_2, z_1, z_2\rangle$ is an elementary abelian group on which the extraspecial 
group $E = \langle x_1,x_2,z\rangle$ of order $p^3$ and exponent $p$ for $p$ odd and dihedral for $p=2$ acts with 
kernel $\langle z\rangle$. Then $P = EA_0$ has order $p^7$, $A = \langle z\rangle A_0$ is an elementary abelian normal 
subgroup of order $p^5$ and $P' = \Z P = \langle z, z_1,z_2\rangle$ has order $p^3$.

Arguing in the same way as in the proof of Proposition~\ref{l2odd} we infer that $[A_0,x] = \langle z_1,z_2\rangle$ 
and hence $\ord{C_{A_0}(x)} = p^2$ for all $x \in P \setminus A$. Consequently, $C_A(x) = \Z P$ for all $x \in P \setminus A$.
This implies that $C_P(a) = A$ for all $a \in A \setminus \Z P$ and also that $C_P(x) = \langle x\rangle \Z P$ for all 
$x \in P \setminus A$.

From this it follows that $C_P(B) = A$ for all subgroups $B$ of $A$ properly containing $\Z P$. Moreover, if $B$ is a subgroup  
of order at most $p^5$ not contained in $A$, $x \in B \setminus A$, then $\ord{C_P(B)} \le \ord{C_P(x)} = p^4$. Finally, 
if $\ord B = p^6$ then, by order considerations, $A \cap B$ is not contained in $\Z P$; hence there exist 
$x \in B \setminus A$ and $a \in (A \cap B) \setminus \Z P$ whence $C_P(B) \le C_P(x) \cap C_P(a) = \Z P$.

It follows that $m_P(B) \le p^9 < p^{10} = m_P(\Z P) = m_P(A) = m_P(P)$ for all subgroups $B \ge \Z P$ different from $\Z P, A, P$. 
This completes the proof. 
\end{proof}

\section{An Extension Theorem and Chains of arbitrary Length}\label{extend}

Given the existence of $p$-groups with Chermak-Delgado lattices being chains of length 1 or 2 
as shown in the previous section, it is only natural to ask: Is there a way to extend these constructions, 
to obtain $p$-groups with chains of arbitrary length as Chermak-Delgado lattices?  
We prove not only that this extension is possible but, more generally:

\begin{thm}\label{extthm} Let $p$ be a prime and $H$ a $p$-group of class at most 2 such that $H \in \CD H$ and 
$H / \Z H$ is elementary abelian.  

There exists a $p$-group $G$ of class 2, an embedding of $H$ into $G$ (whose image will also be denoted by $H$) 
and a normal subgroup $N$ of $G$ with $N \cap H = 1$, $N\Z H > \Z G$ and $NH < G$ such that $G / \Z G$ is 
elementary abelian and $\CD {G} = \{G, \Z G\} \cup \{ N\tilde{H} \mid \tilde{H} \in \CD H \}$. 
Moreover, $\Z G$ is elementary abelian when $\Z H$ is elementary abelian. \end{thm}

Note that because of $N \cap H = 1$ in the theorem, $ \{ N\tilde{H} \mid \tilde{H} \in \CD H \}$ is isomorphic as 
a lattice to $\CD H$ and since $N\Z H > \Z G$ and $NH < G$, the Chermak-Delgado lattice of $G$ extends $\CD H$ both upwards and downwards by one more subgroup.\\

The construction of the group $G$ depends essentially on the construction of the group $P$ with $\CD P$ a chain 
of length 2 presented in the proof of Proposition~\ref{l2n} in the previous section. We recall that $P$ is 
generated by $x_1,x_2,a_1,a_2,z,z_1,z_2$, all of order $p$, where the non-trivial commutator relations are 
\begin{equation*}
[x_1,x_2] = z, [a_1,x_1] = z_1, [a_1,x_2] = z_1z_2^c, [a_2,x_1] = z_2, [a_2,x_2] = z_1,
\end{equation*}
$c$ suitably chosen. Recall also that $\Z P = \langle z,z_1,z_2 \rangle$, $A = \langle a_1,a_2 \rangle \Z P$ is an 
abelian normal subgroup of $P$, $C_P(x) = \langle x \rangle \Z P$ for all $x \in P \setminus A$ and 
$\CD P = \{\Z P, A, P\}$.We use these facts and the notation in the following construction and also in the proof of the extension theorem.\\

\begin{cnstr} Let $E = \langle e_1,\dots ,e_r \rangle$ be an elementary abelian $p$-group of the same rank, $r$, 
as $H / \Z H$.  Let $v_1,\dots ,v_r$ be such that $\{v_i\Z H \mid 1 \le i \le r \}$ is a basis for $H / \Z H$.  
The group $P = \langle x_1,x_2 \rangle A$ acts on $H \times E$ with $\langle x_2 \rangle A$ in the kernel,  where $x_1$ 
induces a central automorphism through $v_i^{x_1} = v_ie_i$, for $1 \le i \le r$, that centralizes $\Z H \times E$. 
With respect to this action we define $G = (H \times E) \rtimes P$.

Note that $x_1$ induces an automorphism of order $p$ on $H \times E$ unless $H$ is abelian in which case 
$E = 1$ and $G = H \times P$. In this situation the statement of the theorem is trivially true (with this $G$ and $N = A$) 
by \cite[Theorem 2.9]{BW2012}. We therefore assume in the remainder of the section that $H$ is non-abelian. \end{cnstr}

For the proof of the theorem it is helpful to collect some information about centralizers in $G$:

\begin{lem}\label{Gcentralizers} Let $G = (H \times E) \rtimes P$ be as described in the construction above with $H$ non-abelian. 
Let $h \in H$ and $x \in P$. 
	\begin{enumerate}
  \item If $H_0 \subseteq H$ then $C_G(H_0) = C_H(H_0)C_P(H_0)E$;
  
  \noindent if $P_0 \subseteq P$ then $C_G(P_0) = C_H(P_0)C_P(P_0)E$.
  \item $\Z G = \Z H \times E \times \Z P$; $G/ \Z G$ is elementary abelian.
	\item If $x \in P \setminus C_P(H)$ then $C_H(x) = C_H(P) = \Z H$.
	\item If $h \in H \setminus \Z H$ then $C_P(h) = C_P(H) = \langle x_2 \rangle A$.
	\item If $h \in H$ and $x \in C_P(H)$ then $C_G(hx) = C_G(h) \cap C_G(x)$.  
  \item If $h \in H$ and $x \in P \setminus C_P(H)$ then $C_G(hx) = \langle hx \rangle \Z G$.
  \end{enumerate}
\end{lem}

\begin{proof} Clearly $\Z H \times E \times \Z P \le \Z G$ and hence $G/ \Z G$ is elementary abelian. Let $h' \in H$ and $x' \in P$. The fact that $G$ has class 2 and 
$P \cap (H \times E) = 1 = H' \cap [H,P]$ implies
\begin{equation}[hx,h'x'] = 1 \quad \textrm{if and only if} \quad [h,h'] = [x,x'] = 1 
\;\, \textrm{and} \;\, [h,x'] = [h',x] \tag{*}.\end{equation} 
From this, parts (1) and (2) follow immediately. Parts (3) and (4) are consequences of the construction of $G$. Let $x \in C_P(H)$. Then $[h',x] = 1$ and by (*), $h'x' \in C_G(hx)$ if and only if 
$h'x' \in C_G(h) \cap C_G(x)$. Since $E \le \Z G$, part (5) follows. 

Assume now that $x \in P \setminus C_P(H)$. Let $h'x' \in C_G(hx)$. By (*), 
$x' \in C_P(x) = \langle x \rangle \Z P$ as $x \not\in C_P(H) \ge A$. Hence $x' = x^{\alpha}y$ for some integer $\alpha$ with 
$0 \le \alpha \le p-1$ 
and some $y \in \Z P$. Also $x = x_1^{\beta}t$ for some integer $\beta$ with $1 \le \beta \le p-1$ and some $t \in C_P(H)$. 
Hence $x' = x_1^{\alpha \beta}t^{\alpha}y'$ for some $y' \in \Z P$.
If $h = v_1^{\alpha_1}\cdots v_r^{\alpha_r} w$ and $h' = v_1^{\beta_1}\cdots v_r^{\beta_r} w'$ with 
integers $0 \le \alpha_i, \beta_i \le p-1$, $i = 1,\ldots,r$ and $w, w' \in \Z H$ then $[h,x'] = e_1^{\alpha_1 \alpha \beta} \cdots e_r^{\alpha_r \alpha \beta}$ and 
$[h',x] = e_1^{\beta_1 \beta} \cdots e_r^{\beta_r \beta}$. By (*), $[h,x'] = [h',x]$, whence 
$\alpha_i \alpha = \beta_i$, $i = 1,\ldots,r$ as $\beta \mod p \ne 0$. 
Therefore $h' = h^{\alpha}\tilde{w}$ for some $\tilde{w} \in \Z H$. It follows that 
$h'x' \Z G = (hx)^{\alpha} \Z G$. This yields (6). 
\end{proof}

\noindent
{\bf Proof of the Theorem.} Note that $m_G(G) = m_G(\Z G) = m(H)m(P)m(E) = \ord H \ord{\Z H}\ord P \ord{\Z P}\ord E^2$. Let $U \in \CD G$ with $\Z G < U < G$. We show that $m_G(U) = m_G(G)$ and that $U$ is of the form given 
in the statement of the theorem. 

Set $H_0 = \{h \in H \mid hx \in U \textrm{ for some } x \in P\}$. Since $[H,P] \le E \le U$ and $H \cap E = 1$,the set $H_0$ 
is a subgroup of $H$. Let $P_0$ be the image of $U$ under the projection of $G$ onto $P$. Clearly $H_0P_0E$ is a subgroup 
of $G$ containing $U$. Note that 
\begin{equation*}
\ord U = \ord{H \cap U} \ord{P_0} \ord E,
\end{equation*} 
since $(H \cap U)E$ is the kernel of the projection 
restricted to $U$ onto $P_0$. We show now that $P_0 \le C_P(H)$.  

Suppose, to the contrary, that $U$ contains $hx$ where $h \in H$ and $x \in P \setminus C_P(H)$.
Lemma~\ref{Gcentralizers} (6) gives $C_G(U) \leq C_G(hx) = \langle hx \rangle \Z G$. Assume there 
exists $h' \in U$ with $h' \in H \setminus \Z H$.  Since $x \not\in C_P(H) = C_P(h')$ (by Lemma~\ref{Gcentralizers} (4)),  
it follows that $[h',x] \ne 1$ and thus $[hx,h'] \ne 1$.  This implies $C_G(U) < \langle hx \rangle \Z G$, whence 
$C_G(U) = \Z G$.  Therefore $m_G(U) < m_G(G)$, as $U < G$, and so $U \not\in \CD G$, contrary to the hypothesis.

We may therefore assume that $H \cap U = \Z H$.  This and the structure of $C_G(U)$ yield
\begin{equation*}
\ord U \le \ord {\Z H} \ord P \ord E \quad \textrm{and} \quad \ord{C_G(U)} \le p\ \ord{\Z G} = p \ord{\Z H} \ord{\Z P} \ord E.
\end{equation*}
Therefore $m_G(U) \le p \ord {\Z H}^2m(P)m(E)$.  The only way that $m_G(U) \ge m_G(G)$ is if 
$p \ord {\Z H} \ge \ord H$, yet by assumption $H$ is non-abelian.  Hence, contrary to the hypothesis, $m_G(U) < m_G(G)$. Thus $P_0 \le C_P(H)$; the next step is to show that $U = H_0 P_0 E$.

From Lemma~\ref{Gcentralizers} (5) we infer
\begin{equation*}
\begin{aligned}
C_G(U) =& \bigcap\limits_{hx \in U} C_G(hx) = \bigcap\limits_{hx \in U} \left(C_G(h) \cap C_G(x)\right)\\
 			   =& \bigcap\limits_{h \in H_0} C_G(h)\ \cap\ \bigcap\limits_{x \in P_0} C_G(x) = C_G(H_0) \cap C_G(P_0)\\
 				   =&\ C_G(H_0P_0E).
\end{aligned}
\end{equation*}
Therefore $m_G(U) \le m_G(H_0P_0E)$, with equality if and only if $U = H_0P_0E$.  Moreover, 
\begin{equation*}
\begin{aligned}
C_G(U) =& C_G(H_0) \cap C_G(P_0)\\
       =& C_H(H_0)C_P(H_0)E \cap C_H(P_0)C_P(P_0)E\\
       =& (C_H(H_0) \cap C_H(P_0))(C_P(H_0) \cap C_P(P_0))E\\
       =& C_H(H_0)(C_P(H_0) \cap C_P(P_0))E.
\end{aligned}
\end{equation*}
Here we have used Lemma~\ref{Gcentralizers} (1) and $P_0 \le C_P(H)$.

Consequently
\begin{equation*}
m_G(U) = m_H(H_0) \ord {P_0} \ord {C_P(H_0) \cap C_P(P_0)}m(E),
\end{equation*}
yielding $m_G(U) \le m_G(G)$ with equality if and only if $H_0 \in \CD H$ and $m(P) = \ord {P_0} \ord {C_P(H_0) \cap C_P(P_0)}$.  The latter condition is equivalent to $P_0 \in \CD P = \{ \Z P, A, P\}$ and $C_P(P_0) \le C_P(H_0)$. Since $P_0$ 
centralizes $H$, it follows that $P_0 \ne P$. If $P_0 = \Z P$ then $C_P(P_0) \le C_P(H_0)$ implies $H_0 \le C_H(P) = \Z H$ whence 
$U = \Z G$. Thus we conclude that if $U \in \CD G$ with $\Z G < U < G$ then $U = NH_0$ where $H_0 \in \CD H$ and $N = AE \nor G$.  Furthermore, 
$m_G(U) = m_G(G)$.  Since $G, \Z G \in \CD G$ as well, the Chermak-Delgado lattice of $G$ is as claimed. \qedhere

\begin{cor} For any non-negative integer $n$ and any prime $p$ there exists a $p$-group $P$ such that 
$P \in \CD P$ and where $\CD P$ is a chain of length $n$. The group $P$ can be chosen in such a way that both $P/ \Z P$ and 
$\Z P$ are elementary abelian. \end{cor}

\begin{proof} For a chain of length 0 we choose a cyclic group of order $p$. For chain length 1 we choose 
the $p$-group from Proposition~\ref{l1n}. Starting with these groups as $H$ and applying the theorem 
iteratively, using the group constructed in the last step as new group $H$ for the next step, yields the assertion.   
\end{proof}

\bibliographystyle{amsplain}
\bibliography{references}

\end{document}